\documentclass[12pt]{amsart}
\usepackage{amsmath,amsthm,amssymb,enumerate}
\usepackage{comment}
\usepackage{color}
\newcommand{\tree}[2]{\mathcal{T}_{#1, #2}}
\newcommand{\N}[2]{\mathcal{N}_{#1, #2}}
\newcommand{\aut}[1]{\operatorname{Aut}{(#1)}}

\newcommand{\C}{\mathbb{C}}
\newcommand{\bs}{\backslash}

\newcommand{\halg}[2]{\mathcal{H} (#1, #2)}
\newcommand{\innprod}[2]{\langle #1 , #2 \rangle}
\newcommand{\btimes}{\mathbin{\bar{\otimes}}}
\newcommand{\ogr}{\mathcal{O}_{d, k}}

\newcommand{\A}{\mathcal{A}}
\renewcommand{\S}{\mathfrak{S}}
\renewcommand{\phi}{\varphi}
\newtheorem{main}{Theorem}

\newtheorem{theorem}{Theorem}[section]
\newtheorem{lemma}[theorem]{Lemma}
\newtheorem{proposition}[theorem]{Proposition}

\newtheorem{corollary}[theorem]{Corollary}

\newtheorem*{definition}{Definition}

\theoremstyle{remark}
\newtheorem*{remark}{Remark}
\title[On an open subgroup of the Neretin group]{On the type of the von Neumann algebra of an open subgroup of the Neretin group}
\author{Ryoya Arimoto}
\address{RIMS, Kyoto University, \mbox{606-8502} Japan}
\email{arimoto@kurims.kyoto-u.ac.jp}


\begin{document}

\begin{abstract}
The Neretin group $\N{d}{k}$ is the totally disconnected locally compact group 
consisting of almost automorphisms of the tree $\tree{d}{k}$. 
This group has a distinguished open subgroup $\ogr$. 
We prove that this open subgroup is not of type I.
This gives an alternative proof of the recent result of P.-E. Caprace, A. Le Boudec and N. Matte Bon 
which states that the Neretin group is not of type I, 
and answers their question whether $\ogr$ is of type I or not.
\end{abstract}

\maketitle
\section{Introduction}
The Neretin group $\N{d}{k}$ was introduced by Yu. A. Neretin in \cite {ne}  
as an analogue of the diffeomorphism group of the circle.
This group $\N{d}{k}$ consists of almost automorphisms of the tree $\tree{d}{k}$ 
and is a totally disconnected locally compact Hausdorff group.
It has a distinguished open subgroup $\ogr$; for an accurate definition, see Section 3.1.
Recently, P.-E. Caprace, A. Le Boudec and N. Matte Bon proved that 
the Neretin group $\N{d}{k}$ is not of type I 
by constructing two weakly equivalent but inequivalent irreducible representations of $\N{d}{k}$ (\cite{cbb}).
In their paper, they conjectured that the subgroup $\ogr$ of the Neretin group $\N{d}{k}$ is not type I either (\cite[Remark 4.8]{cbb}).
Our main theorem answers their question.
\begin{main}
The group von Neumann algebra of $L(\ogr)$ of the open subgroup $\ogr$ of the Neretin group $\N{d}{k}$ is of type II.
In particular, the open subgroup $\mathcal{O}_{d, k}$ of the Neretin group $\N{d}{k}$ is not of type I.
\end{main}
This theorem gives an alternative proof of the fact that the Neretin group $\N{d}{k}$ is not of type I, 
since the type I property is inherited to open subgroups.
In the proof of our main theorem, 
we construct a nontrivial central sequence in the corner of the group von Neumann algebra $L(\ogr)$.


\subsection*{Acknowledgements}
The author would like to express his deep gratitude to his supervisor, Professor Narutaka Ozawa, for his support and providing many insightful comments.
This work was supported by JST SPRING, Grant Number JPMJSP2110 and by JSPS KAKENHI, Grant Number 20H01806.

\section{Preliminaries}
\subsection{Von Neumann algebras}
We refer the reader to \cite{d} for basics about von Neumann algebras.
We review several topologies we use.
Let $H$ be a separable Hilbert space.
For $\xi \in H$, seminorms $p_{\xi} , 
p_{\xi}^*$ on $B(H)$ are defined by $p_{\xi}(x) = \| x \xi \|$ and $p_{\xi}^* (x) = \| x^* \xi \|$.
The topology defined by these seminorms $\{ p_{\xi} \mid \xi \in H \} \cup \{ p_{\xi}^* \mid \xi \in H \}$ on $B(H)$ is called \textbf{strong-$*$ operator topology}.
For $\{ \xi _n \} \in \ell ^2 \otimes H = \{ \{ \xi _n \} \mid \xi _n \in H , \sum _{n=1}^{\infty} \| \xi _n \| ^2 < \infty \}$ , 
seminorms $q_{\{ \xi _n \} } , q_{\{ \xi _n \} }^*$ are defined 
by $q_{\{ \xi _n \} } (x) = ( \sum_{n=1}^{\infty} \| x \xi_n \| ^2 )^{\frac{1}{2}}$ 
and $q_{\{ \xi _n \} }^* (x) = ( \sum_{n=1}^{\infty} \| x^* \xi_n \| ^2 )^{\frac{1}{2}}$.
The topology defined by these seminorms $\{ q_{\{ \xi _n \} } \mid \{ \xi _n \} \in \ell ^2 \otimes H \} \cup \{ q_{\{ \xi _n \} }^* \mid \{ \xi _n \} \in \ell ^2 \otimes H \}$ on $B(H)$ is called \textbf{ultrastrong-$*$ topology}.
Note that these two topologies coincide on bounded subsets of $B(H)$.

We also review definitions of types of von Neumann algebras.
A von Neumann algebra $M$ is of \textbf{type I} 
if it is isomorphic to $\prod _{j \in J} \mathcal{A}_j \btimes B(H_j) $ 
for some set $J$ of cardinal numbers, 
where $\mathcal{A} _j$ is an abelian von Neumann algebra 
and $H_j$ is a Hilbert space of dimension $j$.
A von Neumann algebra $M$ is of \textbf{type II$_\mathbf{1}$} 
if it has no nonzero summand of type I and 
there exists a separating family of normal tracial states.
A von Neumann algebra $M$ is of \textbf{type II$_\mathbf{\infty}$} 
if it has no nonzero summand of type I or $\mathrm{II_1}$ 
but there exists an increasing net of projections $\{ p_i \} _{i \in I} \subset M$ converging strongly to $1_M$ 
such that $p_i M p_i$ is of type $\mathrm{II}_1$ for every $i \in I$. 
A von Neumann algebra $M$ is of \textbf{type II} 
if it is a direct sum of a type II$_1$ and a type II$_{\infty}$ von Neumann algebra.
A von Neumann algebra $M$ is of \textbf{type III} 
if it has no nonzero summand of type I, $\mathrm{II_1}$ or $\mathrm{II_{\infty}}$.
Every von Neumann algebra $M$ has a unique decomposition $M \cong M_{\mathrm{I}} \oplus M_{\mathrm{II}} \oplus M_{\mathrm{III}}$ 
where $M_{\mathrm{I}} , M_{\mathrm{II}} , M_{\mathrm{III}}$ are of type I, type II, type III respectively.

We review types of von Neumann algebras from the perspective of central sequences 
and obtain a criterion of having no nonzero type I summand.

\begin{definition}
Let $M$ be a separable von Neumann algebra.
A \textbf{central sequence} of $M$ is a sequence $\{ u_n \}$ of unitary elements in $M$ 
such that $[x , u_n]$ converges to $0$ in the ultrastorong-$*$ topology for all $x \in M$.
A central sequence $\{ u_n \}$ of $M$ is \textbf{trivial} 
if there exits a sequence $\{ z_n \}$ of unitary elements of the center of $M$ 
such that $u_n - z_n$ converges to $0$ in the ultrastorong-$*$ topology.
\end{definition}

\begin{remark}
A sequence $\{ u_n \}$ of unitary elements in $M$ is a central sequence 
if and only if there exists $M_0 \subset M$ such that $M_0 '' = M$ and 
for all $x \in M_0$, $[x , u_n] \to 0$ in the ultrastrong-$*$ topology.
\end{remark}

A. Connes showed that any type I factor has no nontrivial central sequence (\cite[Corollary 3.10]{co}) and 
this fact can be easily extended to type I von Neumann algebras.

\begin{lemma}\label{keylem}
Let $M$ be a separable von Neumann algebra.
If $M$ is of type I, then every central sequence of $M$ is trivial.
\end{lemma}

\begin{proof}
We may assume that $M$ is isomorphic to $\A \btimes B(H)$ 
for some separable abelian von Neumann algebra $\A$ and some separable Hilbert space $H$.
Let $\{ u_n \}$ be a central sequence in $M$.
Take some unit vector $\eta _0 \in H$ and let $p \in B(H)$ be the projection onto $\C \eta _0$.
Then there exist $a_n \in \A$ 
such that $(1 \otimes p) u_n (1 \otimes p) = a_n \otimes p \in \A \btimes pB(H)p \cong \A \btimes \C p$.
Since $\A$ is abelian, there exists a unitary element $v_n \in \A$ such that $a_n = v_n |a_n|$.
We will show $u_n - v_n \otimes 1 \to 0$ in the strong-$*$ topology.
First, we will show $u_n - a_n \otimes 1 \to 0$ in the strong-$*$ topology.
Fix a faithful representation $\A \subset B(K)$ and take $\xi \in K , \eta \in H$ arbitrarily.
Then, for sufficiently large $n$,
\begin{align*}
u_n (\xi \otimes \eta) &\approx (1 \otimes (\eta \otimes \eta _0 ^*)) u_n (\xi \otimes \eta _0) \\
&= (1 \otimes (\eta \otimes \eta _0 ^*)) (a_n \otimes p) (\xi \otimes \eta _0) \\
&= (a_n \otimes 1) (\xi \otimes \eta)
\end{align*}
where $\eta \otimes \eta _0 ^*$ is a Schatten form; 
$\eta \otimes \eta _0 ^* (\zeta) = \innprod{\zeta}{\eta _0} \eta$.
Similarly, one has $u_n^* (\xi \otimes \eta) \approx (a_n^* \otimes 1) (\xi \otimes \eta)$ for sufficiently large $n$.
Finally, we should prove $|a_n| \to 1$ in $\A$ in the ultrastrong-$*$ topology; 
if this holds, then $a_n \otimes 1 - v_n \otimes 1 = v_n ((|a_n| - 1) \otimes 1) \to 0$ in the ultrastrong-$*$ topology.
Since $t \mapsto \sqrt{t \vee 0 } $ is a linear growth function, 
it suffices to prove $a_n^* a_n \to 1$ in the strong-$*$ topology.
For arbitrary $\xi \in K$,
\begin{align*}
\| a_n^*a_n \xi - \xi \| &= \| (a_n^* a_n \otimes p) \xi \otimes \eta _0 - \xi \otimes \eta _0 \| \\
&= \| (1 \otimes p) u_n^* (1 \otimes p) u_n (1 \otimes p) \xi \otimes \eta _0 - \xi \otimes \eta _0 \| \\
&\to 0.
\end{align*}
Therefore, a central sequence $\{ u_n \}$ in $M$ is trivial.
\end{proof}

\begin{lemma}\label{notypeIsummand}
Let $M$ be a separable von Neumann algebra.
Suppose there exists a faithful normal state $\phi$ and two central sequences $\{ u_n \}, \{ v_n \}$ 
such that $\phi ( (u_n v_n u_n^* v_n^*)^k)$ converges to $0$ for every $k \in \mathbb{Z} \setminus \{ 0 \}$.
Then $M$ has no nonzero type I summand.
\end{lemma}

\begin{proof}
For simplicity, we write $u_n v_n u_n^* v_n^*$ as $w_n$.
Note that for every $f \in C(\mathbb{T})$, $\phi (f(w_n)) \to \int_{\mathbb{T}} f(z) \, dz$ where $\mathbb{T} = \{ z \in \C \mid | z | = 1 \}$, 
since trigonometric polynomials are dense in $C(\mathbb{T})$.
Let $p \in M$ be a central projection such that $pM$ is of type I.
Since every central sequence in a type I von Neumann algebra is trivial 
and $\{ p u_n \}$ and $\{ p v_n \}$ are central sequences in $pM$, 
$p w_n$ converges to $p$ in the ultrastrong-$*$ topology.
Then for every $f \in C(\mathbb{T})$, $\phi (p f(w_n)) \to \phi (p) f(1)$.
Take $\varepsilon > 0 $ arbitrarily and $f \in C(\mathbb{T})$ such that $f \geq 0$, $f(1) = 1$ and $\int _\mathbb{T} f(z) \, dz < \varepsilon$.
Then $\phi (f(w_n)) \geq \phi (p f(w_n))$, so $\phi (p) \leq \int _{\mathbb{T}} f(z) \, dz < \varepsilon$.
Since $\varepsilon$ is arbitrary, $\phi (p) = 0$, i.e., $p = 0$.
Therefore $M$ has no nonzero type I summand.
\end{proof}

\subsection{Hecke algebras}
We refer the reader to \cite{klq} and \cite{lln} for definitions and basic properties of Hecke algebras.

Suppose $(G,H)$ is a Hecke pair and $H \bs G$ is a discrete space.
Then the Hecke algebra $\halg{G}{H}$ acts on $\ell ^2 (H \bs G)$ from left; 
define $\lambda \colon \halg{G}{H} \to B(\ell ^2(H \bs G))$ by 
\[
\left[ \lambda (f) \xi \right] (Hx) = \sum_{Hy \in H \bs G} f(Hxy^{-1}) \xi (Hy)
\]
for $f \in \halg{G}{H}$ and $\xi \in \ell ^2 (H \bs G)$.
We may omit $\lambda$ and write $\halg{G}{H} \subset B(\ell^2(H \bs G))$.

Let $\rho \colon G \to B(\ell^2 (H \bs G))$ be the right quasi-regular representation defined by $[ \rho _s \xi ] (x) = \xi (xs)$.
One can easily check that $\halg{G}{H} \subset \rho (G) '$.
Moreover, one has $\halg{G}{H}'' = \rho (G) '$ (see \cite{ad} Theorem 1.4.).
The unit vector $\delta _H \in \ell ^2 (H \bs G)$ is a separating vector for $\halg{G}{H}$, since $\delta _H$ is a $\rho(G)$-cyclic vector.
Moreover, if $R(x) = R(x^{-1})$ for every $x \in G$, then it is not hard to see that $\delta _H$ is a tracial vector, 
i.e., the vector state associated with $\delta _H$ is a trace on $\lambda (\halg{G}{H})$.
In particular, the vector state $x \mapsto \innprod{x \delta _H}{\delta _H}$ is a faithful tracial state 
of $\halg{G}{H}$ for a unimodular locally compact group $G$ and its compact open subgroup $H$.

For a finite group $G$ and its subgroup $H \leq G$, 
note that the Hecke algebra $\halg{G}{H}$ is identical to $p_H \C [G] p_H$ 
where $p_H = \frac{1}{|H|} \sum _{h \in H} h \in \C [G]$ is a projection (see \cite[Corollary 4.4]{klq}).


The next proposition is a special case of \cite[Proposition 1.3]{lln}.

\begin{proposition}\label{llnprop}
Let $G$ be a finite group acting on a finite group $V$, 
and let $\Gamma$ be a subgroup of $G$ leaving a subgroup $V_0$ of $V$ invariant.
Then we have a canonical embedding $\halg{V}{V_0}^{\Gamma} \hookrightarrow \halg{V \rtimes G}{V_0 \rtimes \Gamma}$.
Moreover, the canonical traces are consistent with this embedding. 
\end{proposition}

\begin{proof}
We will prove that there exists a canonical, trace preserving embedding $(p_{V_0} \C [V] p_{V_0})^{\Gamma} \hookrightarrow p_{V_0 \rtimes \Gamma} \C [V \rtimes G] p_{V_0 \rtimes \Gamma}$ 
where $p_H = \frac{1}{|H|} \sum _{h \in H} h$ for a subgroup $H$. 
Since $\Gamma$ leaves $V_0$ invariant, $p_{V_0}$ commutes with every element of $\Gamma$ in $\C [V_0 \rtimes \Gamma]$.
In particular, $p_{V_0}$ commutes with $p_\Gamma$ 
and $p_{V_0 \rtimes \Gamma} = p_{V_0} p_{\Gamma} = p_{\Gamma} p_{V_0}$. 
Note that $p_{\Gamma}$ commutes with every element in $\C [V] ^{\Gamma}$. 
Therefore, multiplication with $p_{\Gamma}$ is a $*$-homomorphism 
from $(p_{V_0} \C [V] p_{V_0})^{\Gamma} \cong p_{V_0} \C [V]^{\Gamma} p_{V_0}$ 
to $p_{V_0 \rtimes \Gamma} \C [V]^{\Gamma} p_{V_0 \rtimes \Gamma} \subset p_{V_0 \rtimes \Gamma} \C [V \rtimes G] p_{V_0 \rtimes \Gamma}$. 
This map preserves the canonical trace, since it is spatially implemented 
by the canonical isometry $W \colon \ell ^2 (V_0 \bs V) \to \ell ^2 ((V_0 \rtimes \Gamma ) \bs (V \rtimes G))$ , 
and $W^* \delta _{V_0 \rtimes \Gamma} = \delta _{V_0}$. 
Since the canonical traces are faithful, this $*$-homomorphism is an embedding.
\end{proof}

\begin{corollary}\label{cor}
In addition to the assumptions of Proposition $\ref{llnprop}$, suppose $G$ leaves $V_0$ invariant. 
Then there is a canonical trace preserving embedding $\halg{G}{\Gamma} \hookrightarrow \halg{V \rtimes G}{V_0 \rtimes \Gamma}$ 
and $\halg{V}{V_0} ^{G} \subset \halg{G}{\Gamma} '$ in $\halg{V \rtimes G}{V_0 \rtimes \Gamma}$.
\end{corollary}

\begin{proof}
The same argument as above shows that the first assertion holds.
To show the second assertion, we identify $\halg{V}{V_0} ^{G} $ and $ \halg{G}{\Gamma}$ with $p_{V_0} \C [V] ^G p_{V_0}$ and $p_{\Gamma} \C [G] p_{\Gamma}$, respectively.
The assertion follows from the fact that $p_{V_0} p_{\Gamma} = p_{\Gamma} p_{V_0}$ 
and $\C [V] ^G \subset \C [G] '$.
\end{proof}

\subsection{Locally compact groups}
In this paper, topological groups are assumed to be Hausdorff.
Let $G$ be a locally compact second countable group and $\mu$ be its left Haar measure.
The \textbf{left regular representation} of $G$ is a unitary representation $\lambda \colon G \to \mathcal{U} (L^2(G))$ 
defined by $(\lambda _g f) (h) = f(g^{-1}h)$ for $f \in L^2(G)$ 
where $L^2(G)$ is a Haar square integrable functions on $G$.
The von Neumann algebra $\{ \lambda _g \mid g \in G \} '' \subset B(L^2(G))$ is called the group von Neumann algebra.
The representation $\lambda$ extends to a representation of $L^1(G)$: 
$\lambda (f) g = f * g$ for $f \in L^1(G)$ and $g \in L^2(G)$.

A unitary representation $( \pi , H)$ of $G$ is called of \textbf{type I} 
if the associated von Neumann algebra $\pi (G) '' \subset B(H)$ is of type I.
A locally compact group $G$ is called of \textbf{type I} if all its unitary representations are of type I.
See \cite[Chapter 6, 7]{bh} for more details and properties of type I groups.

\section{Neretin groups}
Let $d , k \geq  2$ be integers and $\tree{d}{k}$ be a rooted tree 
such that the root has $k$ adjacent vertices 
and the others have $d+1$ adjacent vertices.
An \textbf{almost automorphism} of $\tree{d}{k}$ is a triple $(A, B, \phi)$
where $A, B \subset \tree{d}{k}$ are finite subtrees containing the root 
with $| \partial A| = | \partial B |$ 
and $\phi \colon \tree{d}{k} \setminus A \rightarrow \tree{d}{k} \setminus B$ is an isomorphism.
The \textbf{Neretin group} $\N{d}{k}$ is the quotient of the set of all almost automorphisms by the relation which identifies two almost automorphisms $(A_1, B_1, \phi_1), \, (A_2, B_2, \phi_2)$ if there exits a finite subtree $\tilde{A} \subset \tree{d}{k}$ containing the root 
such that $A_1 , \, A_2 \subset \tilde{A}$ 
and $\phi_1 | _{\tree{d}{k} \setminus \tilde{A}} = \phi_2 | _{\tree{d}{k} \setminus \tilde{A}}$.
One can easily check that $\N{d}{k}$ is a group.

Let $d$ be the graph metric on $\tree{d}{k}$, $v_0$ be the root of $\tree{d}{k}$ and $B_n := \{ v \in \tree{d}{k} \mid d(v_0 , v) \leq n \}$ for $n \geq 0$.
Every automorphism of $\tree{d}{k}$ leaves $B_n$ invariant.
For each $n \geq 0$, $\mathcal{O}_{d, k}^{(n)}$ denotes the subgroup consisting of automorphisms on $\tree{d}{k} \setminus B_n$ 
and we write $\mathcal{O}_{d, k} := \bigcup_{n=0}^{\infty} \mathcal{O}_{d, k}^{(n)}$.
Each $\mathcal{O}_{d, k} ^{(n)}$ is a subgroup of $\N{d}{k}$ containing $K := \aut{\tree{d}{k}}$.
Let $V_n : = \partial B_n = \{ v \in \tree{d}{k} \mid d(v , v_0) = n \}$.
Note that $\ogr ^{(n)} \cong \aut{\tree{d}{d}} \wr \S_{|V_n|} = \aut{\tree{d}{d}} ^{|V_n|} \rtimes \S_{|V_n|}$ 
and $\mathcal{O}_{d, d} ^{(l)} \wr \S _{|V_n|} < \ogr ^{(n+l)}$.

The Neretin group $\N{d}{k}$ admits a totally disconnected locally compact group topology 
such that the inclusion map $K \hookrightarrow \N{d}{k}$ is continuous and open (\cite[Theorem 4.4]{gl}).
The Neretin group $\N{d}{k}$ is compactly generated and simple; 
see \cite{gl}.

The group $\ogr$ is an open subgroup of $\N{d}{k}$.
It is unimodular and amenable since $\ogr$ is a increasing union $\bigcup _{n=1}^{\infty} \ogr ^{(n)}$ of its compact subgroups.

\section{Proof of Theorem}
We normalize the Haar measure $\mu$ on $\ogr$ so that $\mu (K) = 1$. 
Let $p = \lambda (\chi _K)$ be the projection onto the subspace of left $K$-invariant functions.
This subspace can be identified with $\ell ^2 (K \bs \ogr)$.
The  Hecke algebra $\halg{\ogr}{K} \subset B(\ell^2 (K \bs \ogr))$ is a dense subalgebra of the corner $pL(\ogr)p \subset B(\ell^2(K \bs \ogr))$ 
with respect to the weak operator topology.
We will show that $p L(\ogr) p$ is of type II.

Since $K$ acts on $V_n$, 
there exists a canonical group homomorphism $K \to \aut{V_n} \cong \S_{|V_n|}$.
The range of this homomorphism is denoted by $P_n = \aut{B_n} < \S _{|V_n|}$.
Similarly, let $Q_n$ be the range of the canonical group homomorphism $\aut{\tree{d}{d}} \to \aut{ W_n }$, 
where $W_n$ is the subset $\{ v \in \tree{d}{d} \mid d(v , v_0) = n \}$ of $\tree{d}{d}$.
One has $\halg{\mathcal{O}_{d, k}}{K} \cong \cup _{n=1}^{\infty} \halg{\mathcal{O}_{d, k}^{(n)}}{K}$ and $\halg{\mathcal{O}_{d, k}^{(n)}}{K} \cong \halg{\S_{|V_n|}}{P_n}$.
We use this identification freely.
For finite groups $G_1, G_2$ and their subgroups $H_i < G_i$, $\halg{G_1}{H_1} \otimes \halg{G_2}{H_2} \cong \halg{G_1 \times G_2}{H_1 \times H_2}$. 
Proposition $\ref{llnprop}$ for $G=\S _{|V_n|} , \Gamma = P_n , V = \S _{d^l}^{|V_n|} , V_0 = Q_l^{|V_n|}$ implies
\begin{align*}
((\halg{\mathcal{O}_{d, d}^{(l)}}{\aut{\tree{d}{d}}})^{\otimes |V_n|})^{P_n} &\cong (\halg{\S _{d^l}^{|V_n|}}{Q_l^{|V_n|}})^{P_n} \\
&\hookrightarrow \halg{\S _{d^l}^{|V_n|} \rtimes \S _{|V_n|}}{Q_l^{|V_n|} \rtimes P_n} \\
&= \halg{\S _{d^l}^{|V_n|} \rtimes \S _{|V_n|}}{P_{n+l}} \\
&\subset \halg{\S_{|V_{n+l}|}}{P_{n+l}}
\end{align*}
for $l \in \mathbb{N}$.
Moreover, Corollary $\ref{cor}$ implies $(\halg{\S _{d^l}^{|V_n|}}{Q_l^{|V_n|}})^{\S _{|V_n|}} \subset \halg{\S _{|V_n|}}{P_n}'$.
Since $\halg{\mathcal{O}_{d, d} ^{(l)}}{\aut{\tree{d}{d}}} \cong \halg{\S _{d^l}}{Q_l}$ and $(\S _{d^l} , Q_l)$ is not a Gelfand pair for $l \geq 3$ (see \cite[Theorem 1.2]{gm}), 
$\halg{\mathcal{O}_{d, d} ^{(3)}}{\aut{\tree{d}{d}}}$ is noncommutative.

Let $\tau$ be the vector state associated with $\delta _{K} \in \ell ^2 (K \bs \ogr)$.
This is a trace, since $\ogr$ is a unimodular locally compact group and $K$ is its compact open subgroup.
Note that $\tau (x^{\otimes |V_n|}) = (\tau (x))^{|V_n|}$ for $x \in \halg{\mathcal{O}_{d, d} ^{(3)}}{\aut{\tree{d}{d}}}$ where $\tau$ also denotes the canonical trace on $\halg{\mathcal{O}_{d, d} ^{(3)}}{\aut{\tree{d}{d}}}$.
Since $\halg{\mathcal{O}_{d, d} ^{(3)}}{\aut{\tree{d}{d}}}$ is a non-commutative finite dimensional algebra, 
there exist two unitaries $u , v \in \halg{\mathcal{O}_{d, d} ^{(3)}}{\aut{\tree{d}{d}}}$ 
such that $| \tau ((u^*v^*uv)^k) | < 1$ and $| \tau ((v^*u^*vu)^k) | < 1$ for all $k \in \mathbb{Z} \setminus \{ 0 \}$.
Set $u_n := u^{\otimes |V_n|} \in \halg{\ogr ^{(n)}}{K}'$ and $v_n := v^{\otimes |V_n|} \in \halg{\ogr ^{(n)}}{K}'$.
Then for every $x \in \halg{\ogr}{K}'' = pL(\ogr)p$, $ [x,u_n] \to 0$ and $[x,v_n] \to 0$ in the ultrastrong-$*$ topology.
Thus $\{ u_n \}$ and $\{ v_n \}$ are central sequences.
In addition, $\tau ((u_n v_n u_n^* v_n ^*)^k ) = \tau ( (u v u^* v ^*)^k )^n \to 0$ as $n \to \infty$ for every $k \in \mathbb{Z} \setminus \{ 0 \}$.
So by Lemma $\ref{notypeIsummand}$, $p L(\ogr) p$ has no nonzero type I summand and it is of type II.

Let $K_n := \{ \phi \in K \mid \phi | _{B_n} = \mathrm{id} _{B_n} \}$ 
and $p_n := \frac{1}{\mu (K_n)} \lambda (\chi _{K_n}) \in L(\ogr)$.
Then $\{ p_n \}$ converges $1_{L(\ogr)}$ in the strong operator topology.
Applying the same argument as above to $p_n L(\ogr) p_n$, one finds that $p_n L(\ogr) p_n$ is of type II.
Therefore $L(\ogr)$ is of type II.


\end{document}